\documentclass[11pt]{amsart}
\usepackage{fourier}
\usepackage{amscd}
\usepackage{amsthm}
\usepackage[centertags]{amsmath}
\usepackage{amsfonts}
\usepackage{newlfont}
\usepackage{graphicx}
\usepackage[usenames]{color}
\usepackage{amsfonts, amssymb,accents}
\usepackage{mathrsfs}
\usepackage{latexsym}
\usepackage{bbold}
\usepackage{enumitem}
\usepackage{tikz}\usepackage{verbatim}
\usepackage{tabulary,tabularx}
\usepackage[all]{xy}
\usepackage{tikz}
\usetikzlibrary{positioning}
\usetikzlibrary{shapes.multipart}
\usepackage[colorlinks=true,linkcolor=colorref,citecolor=colorcita,urlcolor=colorweb]{hyperref}
\definecolor{colorcita}{RGB}{21,86,130}
\definecolor{colorref}{RGB}{5,10,177}
\definecolor{colorweb}{RGB}{177,6,38}

\newtheorem{theorem}{Theorem}[section]

\newtheorem{proposition}[theorem]{Proposition}

\newtheorem{lemma}[theorem]{Lemma}
\newtheorem{remark}[theorem]{Remark}
\newtheorem{definition}[theorem]{Definition}

\newcommand{\zN}{\mathbb N}

\renewcommand{\H}{H(\mathbb{C})}
\newcommand{\C}{\mathbb C}

\newcommand{\sub}{\subseteq}

\newcommand{\f}{\frac}

\newcommand{\F}{\mathcal {F}}

\renewcommand{\l}{\left(}
\renewcommand{\r}{\right)}

    \title{Disjoint hypercyclicity, Sidon sets and weakly mixing operators }

\begin{document}
	
	
	

\begin{abstract}
We prove that a finite set of natural numbers $J$ satisfies that $J\cup\{0\}$ is not Sidon if and only if for any operator $T$, the disjoint hypercyclicity of $\{T^j:j\in J\}$ implies that $T$ is weakly mixing. As an application we show the existence of a non weakly mixing operator $T$ such that $T\oplus T^2\ldots \oplus T^n$ is hypercyclic for every $n$.
\end{abstract}

\author{Rodrigo Cardeccia}
\address{Instituto Balseiro, Universidad Nacional de Cuyo – C.N.E.A. and CONICET, Av.
Bustillo 9500, San Carlos de Bariloche, R8402AGP, Rep\'ublica Argentina} \email{rodrigo.cardeccia@ib.edu.ar} 

\thanks{Partially supported by ANPCyT PICT 2015-2224, UBACyT 20020130300052BA, PIP 11220130100329CO and CONICET}

\maketitle
\section { Introduction }
Let $X$ be a Banach space. An operator $T:X\to X$ is said to be hypercyclic provided that there is a vector $x\in X$ such that its orbit $Orb_T(x):=\{T^n(x):n\in \zN\}$ is dense in $X$. In that case, $x$ is said to be a hypercyclic vector for $T$. The study of hypercyclic operators had a lively development in the last decades. See for example the books \cite{BayMat09,GroPer11} on the subject.



 A linear operator is called weakly mixing provided that $T\oplus T:X\oplus X\to X\oplus X$ is hypercyclic. In the topological setting, it is simple to show examples of hypercyclic maps that are not weakly mixing, for instance, any irrational rotation of the torus. However, in the linear setting, things get more interesting as weak mixing is equivalent to the Hypercyclicity Criterion, which is the simplest way to prove that a given operator is hypercyclic. Despite its intricate form it is very simple to use. A linear operator is said to satisfy the hypercyclicity criterion provided that there are dense sets $D_1,D_2\sub X$, a sequence $(n_k)_k$ and applications $S_{n_k}$ such that
 \begin{enumerate}
     \item $T^{n_k}(x)\to 0$ for every $x\in D_1$,
     \item $S_{n_k}(y) \to 0$ and $T^{n_k}S_{n_k}(y)$ for every $y\in D_2$.
     
 \end{enumerate}
 
 Usually, the hypercyclicity criterion is equivalent or implied by a regularly assumption on the operator plus hypercyclicity. For example, hypercyclic operators having a dense set of vectors with bounded orbits and frequently hypercyclic operators are weakly mixing.  The existence of a non weakly mixing but hypercyclic operator was an open question for many years. It was posed by Herrero in the $T\oplus T$ form in 1992 \cite{Herrero92} and solved affirmatively by De La Rosa and Read in 2006 \cite{dlRRea09}. Later on, Bayart and Matheron constructed other examples in  spaces such as $\H$ or $\ell_p$ \cite{BayMat07,BayMat09Non}.  
 A  "natural" example of a hypercyclic operator that does not satisfy the hypercyclicity criterion is still unknown.


A natural question that arises is whether the hypercyclicity of $T\oplus T^2\ldots \oplus T^n$ for every $n$ implies that $T$ is weakly mixing.

\textbf{Question A.}
Let $T$ be an operator such that $T\oplus T^2\ldots \oplus T^n$ is hypercyclic for every $n$. Does $T$ satisfy the hypercyclicity criterion?

Bayart and Matheron's construction of a non weakly mixing but hypercyclic operator invites to consider disjoint hypercyclic operators. A finite set of operators $\{T_j: j\in J\} $ is called disjoint hypercyclic provided that there is a vector $x\in X$ such that $\bigoplus_{j\in J} x$ is a hypercyclic vector for $\bigoplus_{j\in J} T_j$ and it is called disjoint transitive provided that for every nonempty set $U$ and every family of nonempty open sets $\{V_j:j\in J\}$, there is $n$ such that $U\cap \bigcap_{j\in J} T^{-n}(V)\neq \emptyset$. The first to study disjoint hypercyclic operators were Bernal-Gonz\'alez and B\'es and Peris  \cite{Ber07,BesPer07}. Since then, the theory of disjoint hypercyclicity has had a huge impact. We now know that there are disjoint hypercyclic operators that are not disjoint transitive \cite{SanShk14}, the are disjoint weakly mixing operators that fail to satisfy the disjoint hypercyclicity criterion \cite{SanShk14}, there are disjoint hypercyclic operators in every infinite dimensional and separable Banach space \cite{Shk10disjoint}, there are mixing operators that are not disjoint mixing \cite{BesMarShk12}, etc.

Thus, a related question to Question A is the following.

\textbf{Question B.} Let $T$ be an operator such that  $\{T^j: 1\leq j \leq n\}$ is disjoint hypercyclic for every $n$. Does $T$ satisfy the hypercyclicity criterion? More generally, for which subsets of the natural numbers does it follow that if $\{T^j:j\in J\}$ is  disjoint hypercyclic then $T$ is weakly mixing?

The study of these questions leads to a surprising connection with the family of Sidon sets of the natural numbers: In Theorem  \ref{main theorem} we will give a complete answer to Question B by proving that $J\cup\{0\}$  is Sidon if and only if there is a non weakly mixing operator $T$ such that $\{T^j:j\in J\}$ is disjoint hypercyclic. As a corollary, we answer Question A by exhibiting a non weakly mixing operator $T$ such that $T\oplus T^2\ldots \oplus T^n$ is hypercyclic for every $n$ (Theorem \ref{product hypercyclic pero no weakly mixing}).

Recall that a subset $A=\{a_i:i\in\zN\}$ of the natural numbers is Sidon provided that all the sums $a_i+a_j$ for $i\leq j$ are different. The study of Sidon sets had a great development on the last century and is a central task in number theory and additive combinatorics. For instance, in 1941  Erd\"os and Tur\'an \cite{ErdTur41} proved that a result of Singer \cite{Sin38} implies that if $S(n)$ denotes the maximal cardinal of a Sidon set contained in $\{1,\ldots,n\}$, then the asymptotic behavior of $S(n)$ is $n^\f{1}{2}$. They also showed that $S(n)\leq n^\f{1}{2}+O(n^\f{1}{4})$. In 1969 Lindstr\"om \cite{Lin69}  proved that for all $n$, $S(n)<n^\f{1}{2}+n^\f{1}{4}+1$
and in 2010 Cilleruelo \cite{Cil10} slightly improved this result by showing that $S(n)<n^\f{1}{2}+n^\f{1}{4}+\f{1}{2}$. The question whether $S(n)<n^\f{1}{2}+o(n^\varepsilon)$ for every $\varepsilon>0$ is still an open problem posed by Erd\"os \cite{Erd94}.
 



The paper is organized as follows. In Section \ref{preliminares} we fix notation and recall some facts about weakly mixing operators and disjoint hypercyclic operators. In Section \ref{disjoint hiper of powers} we answer Question B, by proving that a finite subset $J\sub \zN$ satisfies that  $J\cup\{0\}$ is not Sidon if and only if for every linear operator $T$ such that $\{T^j:j\in J\}$ is disjoint hypercyclic we have that $T$ is weakly mixing (Theorem \ref{main theorem}). Moreover, we construct a non weakly mixing  operator $T$ such that $\{T^j:j\in J\}$ is disjoint hypercyclic for every finite set $J$ such that $J\cup\{0\}$ is Sidon (Theorem \ref{Sidon non weakly}). In Section \ref{Seccion multihiper} we answer Question A, and exhibit a non weakly mixing operator for which $T\oplus T^2\ldots \oplus T^n$ is hypercyclic for every $n$. In Section \ref{syndetically transitive} we study syndetically transitive operators. We prove that a linear operator $T$ is syndetically transitive if and only if $T\oplus S$ is hypercyclic for every weakly mixing operator $S$ (Theorem \ref{caract syndet}) and that a linear operator is piecewise syndetically transitive if and only if $T\oplus S$ is hypercyclic for every syndetically transitive operator $S$ (Theorem \ref{caract piece}). Finally we show the existence of a frequently transitive but non weakly mixing operator (Theorem \ref{frequently transitive non weakly mixing}), which answers a question of \cite[Question 5.12]{BesMen19}.



\section{Preliminaries}\label{preliminares}
Throughout the paper $X$ will denote an infinite dimensional and separable Fr\'echet space and $T:X\to X$ will be a linear operator.

Given a linear operator $T$, $x\in X$ and $U,V$ nonempty open sets, the sets of hitting times $N_T(x,U)$ and $N_T(U,V)$ are defined as
$$N_T(x,U):=\{n\in\zN: T^n(x)\in U\} \quad \text{and}$$
$$N_T(U,V):=\{n\in\zN: T^n(U)\cap V\} .$$

A linear operator is said to be hypercyclic if there is $x\in X$ such that $N_T(x,U)\neq \emptyset$ for every nonempty open set $U$ and transitive if $N_T(U,V)\neq \emptyset$ for every pair of nonempty open sets $U,V$. Given a hypercyclic vector $x$ and nonempty open sets $U,V$, then we can write $N(U,V)$ as 
$$N_T(U,V)=N_T(x,V)- N_T(x,U):=\{m-n: m\in N_T(x,V),N_T(x,U) \text{ and } m\geq n\}. $$
See \cite[Lemma 4.5]{BayMat09} for a proof of this fact.

A linear operator $T$ is said to be weakly mixing provided that $T\oplus T$ is hypercyclic. The weak mixing property admits several well known equivalent formulations. For instance, a nice result due to B{\`e}s and Peris \cite{BesPer99} shows that $T$ is weakly mixing if and only if $T$ satisfies the hypercyclicity criterion if and only if $T$ is hereditarily hypercyclic. The following characterization  \cite[Proposition 1.53]{GroPer11} of weakly mixing operators will be used repeatedly. 
\begin{proposition}\label{carac weakly mixing}
A linear operator is weakly mixing if and only if for every pair of nonempty sets $U,V$ then $N_T(U,U)\cap N_T(U,V)\neq \emptyset$.
\end{proposition}

\begin{definition}
A set of operators $\{T_i:X\to X:i\in I\}$ is said to be disjoint hypercyclic (or $d$-hypercyclic) provided that there is $x\in X$ such that for every family of nonempty open sets $\{U_i:i\in I\}$, there is $n$ such that $T_i^n(x)\in U_i$ for every $i\in I$. In that case we will say that $x$ is a disjoint hypercyclic vector ( or $d$-hypercyclic vector) for $\{T_i:X\to X:i\in I\}$.  
\end{definition}

Similarly, there is a notion of $d$-transitivity.
\begin{definition}
A set of operators $\{T_i:i\in I\}$ is said to be disjoint transitive (or $d$-transitive) provided that for every nonempty open set $U$ and every family of nonempty open sets $\{U_i:i\in I\}$ there is $n$ such that $U\cap \bigcap_{i\in I} T_i^{-1}U_i$ is nonempty.
\end{definition}
It is not difficult to prove that a set of disjoint transitive operators is a set of disjoint hypercyclic operators with a dense set of $d$-hypercyclic vectors \cite[Proposition 2.3]{BesPer07}. However, the converse is false and there are $d$-hypercyclic operators without a dense set of $d$-hypercyclic vectors \cite[Corollary 3.5]{SanShk14}. 

The next proposition seems to be original. It establishes the equivalence between disjoint hypercyclicity and disjoint transitivity for commuting operators.

\begin{proposition}\label{disjoint conmutan}
Let $\{T_i:1\leq i\leq N\}$ be a set of operators such that $T_1$ commutes with $T_i$ for every $2\leq i\leq N$. Then $\{T_i:1\leq i\leq N\}$ is disjoint hypercyclic if and only if it is disjoint transitive. 
\end{proposition}
\begin{proof}
If $\{T_i:1\leq i\leq  N\}$ is disjoint transitive then it is disjoint hypercyclic by \cite[Proposition 2.3]{BesPer07}.

Suppose that $\{T_i:1\leq i\leq N\}$ is disjoint hypercyclic and let
$x$ be a disjoint hypercyclic vector for $\{T_i:1\leq i\leq N\}$. We will prove that, for every $n$, $T_1^n(x)$ is a disjoint hypercyclic vector for $\{T_i:1\leq i\leq N\}$. Let $n\in\zN$,

As the operators commute, it follows that

$$Orb_{T_1\oplus \ldots \oplus T_N}(\bigoplus_{i=1}^N T_1^n(x))=
\bigoplus_{i=1}^N T_1^n \l Orb_{T_1\oplus \ldots \oplus T_N}(\bigoplus_{i=1}^N x)\r.$$

The operator $T_1$ is hypercyclic and hence $T_1^n$ is also hypercyclic \cite[Theorem 6.2]{GroPer11}. This implies that  $Rg(\bigoplus_{i=1}^N T_1^n)$ is dense in $\bigoplus_{i=1}^N X$. On the other hand, $\bigoplus_{i=1}^N  Orb_{T_1\oplus \ldots \oplus T_N}(\bigoplus_{i=1}^N x)$ is also dense in $\bigoplus_{i=1}^N X$.  We conclude that $T_1^n(x)$ is a disjoint hypercyclic vector.

Consider now $U$ and $V_1,\ldots V_N$ nonempty open sets. Since $x$ is a hypercyclic vector for $T_1$, there is $n_1$ such that $y=T_1^{n_1}(x)\in U$. We have just proved that $y$ is a $d$-hypercyclic vector for $\{T_i:1\leq i\leq N\}$. Thus, there is $n$ such that $T_i^n(y)\in V_i$ for every $1\leq i\leq N$. It follows that $y\in U\cap \bigcap_{i=1}^N T_i^{-n}V_i.$
\end{proof}

\section{Disjoint hypercyclicity of powers of an operator and weakly mixing operators}\label{disjoint hiper of powers}
In this section, we prove the main Theorem of the paper. It involves a surprising connection between disjoint hypercyclicity, weakly mixing operators and Sidon sets of the natural numbers. We prove that a finite set of natural numbers $J$ satisfies that $J\cup \{0\}$ is not Sidon if and only if for any operator $T$, the disjoint hypercyclicity of $\{T^j:j\in J\}$ implies that $T$ is weakly mixing (Theorem \ref{main theorem}). Moreover, we construct a non weakly mixing operator $T:\ell_1\to \ell_1$ such that $\{T^{j}:j\in J\}$ is disjoint hypercyclic for every finite set $J$ such that $J\cup\{0\}$ is Sidon (Theorem \ref{Sidon non weakly}). This allows us to generalize Theorem \ref{main theorem} to infinite sets: an infinite set $S\sub\zN$ satisfies that $S\cup\{0\}$ is Sidon if and only if there is a non weakly mixing operator $T:\ell_1\to\ell_1$ such that for every finite subset $J\sub S$ we have that $\{T^j:j\in J\}$ is disjoint hypercyclic (Theorem \ref{infinite Sidon}).

\begin{definition}
A sequence of integers numbers $(j_k)$ (or a set $J=\{j_k:k\in\zN\}$) is  said to be Sidon, provided that all the sums $j_k+j_{k'}$ with  $k\le k'$ are different.
\end{definition}

In this note, we will only consider Sidon subsets of the positive numbers that contain $0$. Thus, for example, $\{2,4\}$ is Sidon but $\{0,2,4\}$ is not Sidon. 

\begin{theorem}\label{main theorem}
Let $J\sub \zN$ be a finite set. Then $J\cup \{0\}$ is Sidon if and only if there exists a non weakly mixing operator $T:\ell_1\to \ell_1$ such that $\{T^j:j\in J\}$ is disjoint hypercyclic.
\end{theorem}

The proof is an immediate consequence of Theorems \ref{Sidon 1} and \ref{Sidon non weakly} below.

\begin{theorem}\label{Sidon 1}
Let $X$ be a Banach space,
$J\sub \zN$ such that $J\cup\{0\}$ is not Sidon and $T:X\to X$ be a linear operator such that \{$T^{j}:j\in J\}$ is disjoint hypercyclic. Then $T$ is weakly mixing.
\end{theorem} 

\begin{theorem}\label{Sidon non weakly}
There exists a non weakly mixing operator $T:\ell_1\to\ell_1$ such that $\{T^{j}:j\in J\}$ is disjoint hypercyclic  for every finite set $J\sub \zN$ such that $J\cup \{0\}$ is Sidon. 
\end{theorem}

We see in particular that the disjoint hypercyclicity of $\{T,T^j\}$ implies that $T$ is weakly mixing if and only if either $j=1$ or $j=2$ (But of course $\{T,T\}$ is never disjoint hypercyclic).

Since a set $S$ is Sidon if and only if every finite subset of $S$ is Sidon, Theorems \ref{Sidon 1} and \ref{Sidon non weakly} also  give a characterization for infinite subsets of the natural numbers.

\begin{theorem}\label{infinite Sidon}
Let $S\sub \zN$. Then $S\cup \{0\}$ is Sidon if and only if there is a non weakly mixing operator $T:\ell_1\to \ell_1$ such that for every finite subset $J\sub S$ we have that $\{T^j:j\in J\}$ is disjoint hypercyclic.
\end{theorem}

\begin{proof}[Proof of Theorem \ref {Sidon 1}]
 Suppose that there are $0\leq j_1\leq j_2 \leq j_3\leq j_4$ such that  $j_1+j_4=j_2+j_3$.
 
 Let $U,V$ be nonempty open sets. We will prove that $N_T(U,U)\cap N_T(U,V)\neq \emptyset.$ By Proposition \ref{carac weakly mixing} this implies that $T$ is weakly mixing.
 
 We will divide the proof into three cases.
 
 \textbf{First case.}
 Suppose that $j_1\neq 0$ and that $j_2<j_3$. 
 Let $x$ be a disjoint hypercyclic vector for $\{T^{j_i}:1\leq i\leq 4\}$. Hence,
 there is $n$ such that $T^{j_in}(x)\in U$ for $i\leq 3$ and $T^{j_4}(x)\in V$. Therefore, 
 $j_4n-j_2n\in N_T(x,V)-N_T(x,U)=N_T(U,V)$ and on the other hand  
 $j_4n-j_2n=j_3n-j_1n\in N_T(x,U)-N_T(x,U)=N_T(U,U)$. 
 
 If $j_1\neq 0$ and $j_2=j_3$ the proof is the same by considering $x$ a disjoint hypercyclic vector for $\{T^{j_1},T^{j_2},T^{j_4}\}$ and $n$ such that $T^{j_in}(x)\in U$ for $i\leq 3$ and $T^{j_4n}(x)\in V$.

 \textbf{Second case.} Suppose  that $j_1=0$ and that $j_2\neq j_3$.  Let $x$ be a disjoint hypercyclic vector for $\{T^{j_2},T^{j_3}, T^{j_4}\}$. 
 
 Let $k>0$ such that $T^{k}(x)\in U$ and notice that $(x,T^{j_3}(x),x)$ is a hypercyclic tuple for $T^{j_2}\oplus T^{j_3}\oplus T^{j_4}$. Indeed, $Id\oplus T^k\oplus Id$ has dense range and $Orb_{T^{j_2}\oplus T^{j_3}\oplus T^{j_4}}(x,T^k(x),x)= Id\oplus T^k \oplus Id\l Orb_{T^{j_2}\oplus T^{j_3}\oplus T^{j_4}}(x,x,x)\r$. Let $n$ such that $T^{j_2n}(x)\in U$, $T^{j_3n+k}(x)\in U$ and $T^{j_4n}(x)\in V$. Therefore $j_4n-(j_3n+k)\in N_T(x,V)-N_T(x,U)=N_T(U,V)$. On the other hand $j_4n-(j_3n+k)=j_2n-k\in N_T(x,U)-N_T(x,U)=N_T(U,U)$.  
 
 \textbf{Final case.}  Suppose that $j_1=0$, $j_2=j_3=j$ and $j_4=2j$. 
 Hence, there is a disjoint hypercyclic vector $x$ for $\{T^{j},T^{2j}\}$. 
 Let $k>0$ such that $T^{jk}(x)\in U$ and $T^{2jk}(x)\in U$. Notice that $(T^{jk}(x),x)$ is a hypercyclic tuple for $T^{j}\oplus T^{2j}$. Indeed, $T^{jk}\oplus Id$ has dense range and $Orb_{ T^{j}\oplus T^{2j}}(T^{jk}(x),x)= T^{jk}\oplus Id \l Orb_{T^{j}\oplus T^{2j}}(x,x)\r$. Therefore, there is $n$ such that $T^{jn+jk}(x)\in U$,  $T^{2jn}(x)\in V$.

 Then $jn-jk=(jn+jk)-2jk\in N_T(x,U)- N_T(x,U)=N_T(U,U)$
 and
 $jn-jk=2jn-(jn+jk)\in N_T(x,V)- N_T(x,U) =N_T(U,V)$. 
 
 This implies that $N_T(U,U)\cap N_T(U,V)\ne\emptyset$ and thus $T$
 is weakly mixing.
 \end{proof}



To prove Theorem \ref{Sidon non weakly} we will use Bayart and Matheron \cite{BayMat09Non} construction of a non weakly mixing and hypercyclic operator. Let us recall briefly their  construction.

Bayart and Matheron's operator is an upper triangular perturbation of a weighted forward shift in $\ell_1$. For a sparse sequence $(b_n)$, a null sequence $(a_n)$, weights $w_n$  and a dense sequence of polynomials $(P_n)$ to be specified, they consider 

$$\begin{cases}
&T(e_i)=
w_ie_{i+1} \text { if } b_{k-1}\leq i <b_{k}-1;\\
&T^{b_k}(e_1)=P_k(T)(e_1)+ \f{e_{b_k}}{a_k}.
\end{cases}
$$

Since $T$ is an upper triangular operator, $e_1$ is a cyclic vector for $T$ and hence $\{P(T)(e_1): P \text { is a polynomial }\}$ is dense in $\ell_1$. Therefore if $\f{1}{a_n}\to 0$ and $(P_n)$ is a dense family of polynomials, it follows that $e_1$ is a hypercyclic vector for $T$. We notice also that $span(Orb_T(e_1))=c_{00}$. The following definition is useful.
\begin{definition}
Let $(P_n)$ be a sequence of polynomials and $(u_n)$ an increasing sequence of positive real numbers. We will say that $(P_n)$ is controlled by $(u_n)$ provided that for every $n$ $deg(P_n)$ and $|P_n|_1$ are both less than $u_n$. 
\end{definition}

\begin{definition}
Let $(\Delta_l)$ be a sequence of natural numbers. An
increasing sequence of natural numbers ($b_n$) is said to be a ($\Delta_l$)-Sidon
sequence if the sets of natural numbers $$J_l :=[b_l,b_l+\Delta_l]\bigcup
\bigcup_{k\leq l}
[b_l + b_k , b_l + b_k + \Delta_l]$$
are pairwise disjoint.
\end{definition}

The following Theorem is deduced from the proof of \cite[Theorem 1.6]{BayMat09Non}.

\begin{theorem}\label{Teo Bay}
Let $\Delta_l\to \infty$ and $(b_n)$ be a $(\Delta_l)$-Sidon sequence. Then there are paramaters $w_n$, $a_n\to \infty$ and $u_n\to \infty$ such that  whenever $(P_n)$ is controlled by $(u_n)$ then the operator $T$ is continuous and not weakly mixing. 
\end{theorem}


\begin{proof}[Proof of Theorem \ref{Sidon non weakly}]
Let $(F_n)_n$ be a collection of  finite sets, with $|F_n|=n$, such that $F_n\cup \{0\}$ is a Sidon set and such that any finite set $F\subset\mathbb N$ such that $F\cup\{0\}$ is Sidon, is contained in  $F_n\cup \{0\}$ for some $n$.  We consider $(j_{n,k})_{0\leq k\le n,n\in\mathbb N}$ such that for every $n$, $(j_{n,k})_{0\le k\le n}$ forms an increasing enumeration of $F_n\cup\{0\}$. Thus, it suffices to show the existence of a Bayart-Matheron operator such that for each $n$, the set of operators $\{T^{j_{n,1}},\dots,T^{j_{n,n}}\}$ is  disjoint hypercyclic.

For a sequence $m_{l,n}$ such that $l\geq n$ to be defined we will consider $b_{l,n,k}$, $1\leq k\leq n\leq l$ such that  $b_{l,n,k}=m_{l,n} j_{n,k}$. Note that $b_{l,n,k}$ is not defined if $k=0$.

The order considered for the tuples $(l,n)$  is lexicographic, that is,
      $(l,n)\leq (l',n')$ if  $l<l'$ or if $l=l'$ and $n\le n'$. The tuples $(l,n,k)$ will also be ordered lexicographically.

We will construct $m_{l,n}$ by induction in $(l,n)$ so that the sets 
$$
J_{l,n,k}:=[b_{l,n,k},b_{l,n,k}+\frac{m_{l,n}}{2}]\bigcup\left(\bigcup_{(l',n',k')\leq (l,n,k)}\left[b_{l,n,k}+ b_{l',n',k'} ,b_{l,n,k}+ b_{l',n',k'}+\f{m_{l,n}}{2}\right]\right)
$$
with $1\le k\le n\le l$
are pairwise disjoint. If so, it will follow by definition that $(b_{l,n,k})$ is a $\Delta_{l,n,k}$-Sidon sequence for $\Delta_{l,n,k}=\f{m_{l,n}}{2}$.
At each inductive step $(L,N)$ we will construct sets $J_{L,N,k}$, $1\leq k\leq N$ so that
\begin{enumerate}
    \item  for $1\leq k\leq N$ the $J_{L,N,k}$ are pairwise disjoint and
    \item for every $1\leq k\leq N$, each $J_{L,N,k}$ is disjoint from $
\bigcup_{(l',n',k')<(L,N,k)} J_{l',n',k'}.$ 
\end{enumerate}

The first step is straightforward because there is a single set. We put $m_{1,1}=1$.

Suppose now that we have constructed $m_{1,1},\ldots m_{L,N}$ such all the sets $J_{l,n,k}$ with $(l,n,k)\leq (L,N,N)$ 
are pairwise disjoint. If $(\tilde L,\tilde N)$ denotes the immediate successor of $(L,N)$ we have to choose $m_{\tilde L, \tilde N}$ such that the $J_{l,n,k}$ are pairwise disjoint for every $(l,n,k)\leq (\tilde L,\tilde N ,\tilde N)$. To do that, we will choose $m_{\tilde L,\tilde N}$ big enough so that for $k\leq \tilde N$ the minimum of $J_{\tilde L,\tilde N,k}$ is greater than the maximum of $\bigcup_{(l,n,k)\leq (L,N,N)} J_{l,n,k}$. 
Then we will use that $(j_{n,k})_k$ is Sidon to show that, for $k\leq \tilde N$, the $J_{\tilde L,\tilde N,k}$ are pairwise disjoint.

Let $m_{\tilde L,\tilde N}$ such that for every $(l',n')\leq (L,N),$
\begin{equation}\label{def m_n}
  m_{\tilde L,\tilde N}>2m_{l',n'}j_{n',n'}+\frac{m_{l',n'}}{2}  .
  \end{equation}

We claim that for every $k\leq \tilde N$, the set $J_{\tilde L,\tilde N,k}$ is disjoint to $J_{l',n',k'}$ for every $(l',n',k')\leq (L,N,N)$. Indeed, 
\begin{align*}
    \min J_{\tilde L,\tilde N,k}&= b_{\tilde L,\tilde N,k}= m_{\tilde L,\tilde N}j_{\tilde N,k}\geq m_{\tilde L,\tilde N}j_{\tilde N,1}\\
    &> 2m_{l',n'}j_{n',k'}+\frac{m_{l',n'}}{2}=
\max J_{l',n',k'}.
\end{align*} 

So, it only remains to prove that $J_{\tilde L,\tilde N,k}$, with $1\leq k\leq \tilde N$, are pairwise disjoint.


Suppose otherwise and let $t\in J_{\tilde L,\tilde N,k_1}\cap J_{\tilde L,\tilde N,k_2}$. Hence, there are $ (L_i',N_i',k_i')\leq (\tilde L,\tilde N,k_i)$ such that for $i=1$ and $i=2$ we have that
$$m_{\tilde L,\tilde N}j_{\tilde N,k_i}+m_{L_i',N_i'}j_{N_i',k_i'}\leq t\leq m_{\tilde L,\tilde N}j_{\tilde N,k_i}+m_{L_i',N_i'}j_{N'_i,k_i'} +\f{m_{\tilde L,\tilde N}}{2}.$$
Note that $k_i'$ may be equal to 0 here. This is the case if $t\in [b_{\tilde L,\tilde N,k_i},b_{\tilde L,\tilde N,k_i}+\frac{m_{\tilde L,\tilde N}}{2}]$.

Therefore
$\displaystyle j_{\tilde N, k_i}+\frac{m_{L_i',N_i'}}{m_{\tilde L,\tilde N}}j_{\tilde N,k_i'}\leq \f{t}{m_{\tilde L,\tilde N}}\leq j_{\tilde N,k_i}+\frac{m_{L_i',N_i'}}{m_{\tilde L,\tilde N}}j_{\tilde N,k_i'} +\f{1}{2}$, $i=1,2$.
Applying \eqref{def m_n} we obtain that if $(L_i',N_i')<(\tilde L,\tilde N)$,   
$$
j_{\tilde N, k_i}\leq \f{t}{m_{\tilde L,\tilde N}}< j_{\tilde N,k_i} +1.
$$
Otherwise, if $(L_i',N_i')=(\tilde L,\tilde N)$, we obtain $$
j_{\tilde N, k_i}+j_{\tilde N,k_i'}\leq \f{t}{m_{\tilde L,\tilde N}}\leq j_{\tilde N,k_i}+j_{\tilde N,k_i'} +\f{1}{2}.
$$
Thus, for example, if $(L_1',N_1')<(\tilde L,\tilde N)$ and $(L_2',N_2')=(\tilde L,\tilde N)$, the above inequalities show that $j_{\tilde N, k_1}=\lfloor \f{t}{m_{\tilde L,\tilde N}}\rfloor=j_{\tilde N, k_2}+j_{\tilde N, k_2'}$.
This is a contradiction because  $\{0,j_{\tilde N,1},\dots,j_{\tilde N,\tilde N}\}$ is Sidon, $k_1\neq k_2$ and $k_2'\leq k_2$.

The other cases 
\begin{itemize}
    \item $(L_1',N_1')<(\tilde L,\tilde N)$ and $(L_2',N_2')<(\tilde L,\tilde N)$,
    \item $(L_1',N_1')=(\tilde L,\tilde N)$ and $(L_2',N_2')<(\tilde L,\tilde N)$,
    \item $(L_1',N_1')=(L_2',N_2')=(\tilde L,\tilde N)$
\end{itemize}
 are similar.

We have proved that the sets $J_{l,n,k}$, $1\le k\le n\le l$ are pairwise disjoint and thus $b_{l,n,k}$ is a $\Delta_{l,n,k}$-Sidon sequence for some $\Delta_{l,n,k}\to \infty$. Therefore, by Theorem \ref{Teo Bay}, there are parameters $w_{l,n,k},\f{1}{a_{l,n,k}}\to 0$ and $u_{l,n,k}\to\infty$ such that whenever $P_{l,n,k}$ is controlled by $u_{l,n,k}$ then $T$ is continuous and not weakly mixing.  

We consider a family of polynomials $P_{l,n,k}=P_{l,k}$ controlled by $u_{l,n,k}$ such that for every $n,$ $\l P_{l,1}\oplus P_{l,2},\ldots \oplus P_{l,n}\r_{l\geq n}$ is dense in $\bigoplus_{k\leq n} \C([x])$. 
To construct this sequence of polynomials just consider $v_{l,k}=\min_{n\in[k,l]} \{u_{l,n,k}\}$ and a dense sequence $(Q_l)_l\subset \bigoplus_{k\in\zN} \C[x]$, where $Q_l=([Q_l]_1,[Q_l]_2,\dots)$,  with the additional property that each $[Q_l]_k$ is controlled by $ v_{l,k}$ whenever $l\geq k$.

The polynomials $P_{l,k}=[Q_l]_k$ satisfy the desired property.

It remains to show that for every $n$ and $k$, $e_1\oplus e_1\ldots \oplus e_1$ is a hypercyclic vector for $T^{j_{n,1}}\oplus T^{j_{n,2}}\ldots \oplus T^{j_{n,n}}$. Indeed, 
\begin{align*}
&(T^{j_{n,1}}\oplus T^{j_{n,2}}\ldots \oplus T^{j_{n,n}})^{m_{l,n}}(e_1\oplus\ldots \oplus e_1)=  (T^{b_{l,n,1}}\oplus T^{b_{l,n,2}}\ldots \oplus T^{b_{l,n,n}})(e_1\oplus\ldots \oplus e_1)\\
&=P_{l,1}(T)(e_1)+ \f{e_{ b_{l,n,1}}}{ a_{l,n,1}}\oplus \ldots \oplus P_{l,n}(T)(e_1)+ \f{e_{ b_{l,n,n}}}{ a_{l,n,n}}.    
\end{align*}

Thus $((T^{j_{n,1}}\oplus T^{j_{n,2}}\ldots \oplus T^{j_{n,n}})^{m_{l,n}}(e_1\oplus\ldots\oplus e_1))_{l\ge n}$
 is dense in $\bigoplus_{1\leq k\leq n} \ell_1$. 
\end{proof}

\begin{remark}
Our definition of $\Delta_l$-Sidon set is slightly different to the originally proposed by Bayart and Matheron in \cite{BayMat09Non}. The reason is that their condition is not strong enough to prove Theorem \ref{Teo Bay}. Indeed, otherwise, we could, using the same techniques that we used to prove Theorem \ref{Sidon non weakly}, construct a non weakly mixing operator such that $T$ and $T^2$ are disjoint hypercyclic. These conditions are incompatible since $\{T,T^2\}$ disjoint hypercyclic implies that $T$ is weakly mixing. 
\end{remark}

\section{A non weakly mixing operator such that $T\oplus T^2\ldots \oplus T^n$ is hypercyclic for every $n$}\label{Seccion multihiper}

In this section,  we exhibit a non weakly mixing operator such that $T\oplus T^2\ldots \oplus T^n$ is hypercyclic for every $n$ (Theorem \ref{product hypercyclic pero no weakly mixing}). The operator is the one defined in Theorem \ref{Sidon non weakly}.

To show that the operator satisfies the desired property, we will study a nice relationship between the disjointness hypercyclicity of $\{T^j:j\in J\}$ and the hypercyclicity of $\bigoplus_{k\in K} T^k$ for some subsets $K\sub J- J$.

The next Proposition characterizes the hypercyclicity of $\bigoplus_{j=1}^n T_j$ for an $n$-tuple of hypercyclic operators.
\begin{proposition}
Let $T_1,\ldots T_N$ be hypercyclic operators such that for every $N+1$-tuple of nonempty open sets $U,V_1,\ldots V_N$ we have that $\bigcap_{i=1}^N N_{T_i}(U,V_i)\neq \emptyset$.  Then $\bigoplus_{i=1}^N T_i$ is hypercyclic.
\end{proposition}
\begin{proof}
Let $U_1,V_1,U_2,V_2\ldots U_N,V_N$ be nonempty open sets. 

Put $W_1=U_1$ and $n_1=0$. By an inductive argument we construct nonempty open sets $W_N\sub W_{N-1}\ldots \sub W_1\sub U_1$ and numbers $n_1,\ldots n_N$ such that for every $i$, $n_i\in N_{T_{i}}(W_{i-1},U_i)$ and $W_i=W_{i-1}\cap T_i^{-n_i}(U_i)$.

Since each $T_i$ is hypercyclic we have that each $T^{-n_i} (V_i)$ is a nonempty open set. 
Let $m\in \bigcap_{i=1}^N N_{T_i} (W_N, T^{-n_i} (V_i))$. We will show that $m\in \bigcap_{i=1}^N N_{T_i}(U_i,V_i)$. For $i=1$ it is clear because $W_N\sub U_1$ and $T^{-n_1}(V_1)=V_1$. If $i>1$, there is $x_i\in W_N$ such that $T_i^m(x_i)\in T_i^{-n_i}(V_i)$. Hence $T_i^{m+n_i}(x_i)\in V_i$. Using that $W_N\sub W_i\sub T_i^{-n_i}(U_i)$ we see that $T_i^{n_i}(x_i)\in U_i$. Therefore $m\in N_{T_i}(U_i,V_i)$.
\end{proof}

In the same way that $\{T^{j}:j\in J\}$ being disjoint hypercyclic implies that $T\oplus T$ is hypercyclic for $J\cup\{0\}$ not Sidon, there are nice relationships between the disjoint hypercyclicity of $\{T^{j}:j\in J\}$ and the hypercyclicity of $\bigoplus_{k\in K} T^{k}$ for some subsets $K\sub J-J$.

\begin{proposition}\label{d-hiper implica sumas hiper}
Let $J\sub \zN$ be a finite set and  $\{T^j:j\in J\}$ be disjoint hypercyclic. Let $(j_2^l)_{1\leq l\leq n}\sub J$ and $(j_1^l)_{1\leq l\leq n}\sub J\cup\{0\}$ such that
\begin{enumerate}
    \item [i)]$j_2^l\neq j_1^{l'}$ for every $1\leq l,l'\leq n$;
    \item   [ii)] $j_2^l\neq j_2^{l'}$ for every $1\leq l< l'\leq n$ and
    \item [iii)]$j_1^l<j_2^l$ for every $1\le l\le n$.
\end{enumerate}
  Then $\bigoplus_{l=1}^n T^{j_2^l-j_1^l}$ is hypercyclic.


\end{proposition}
\begin{proof}
Let $U$ be a nonempty open set and $V_l: 1\leq l\leq n$ be nonempty open sets. We will prove that $\bigcap_{l\leq n} N_{T^{j_2^l-j_1^l}}(U,V_l)
$ is nonempty.

By Proposition \ref{disjoint conmutan} there is
 $x\in U$ a disjoint hypercyclic vector. Let $m\in\zN$ such that for every $1\leq l\leq n$, $T^{j^l_2 m}(x)\in V_l$ and $T^{j_1^lm}(x)\in U$. Therefore, $j_2^lm-j_1^lm\in N_T(x,V_l)-N_T(x,U)=N_T(U,V_l)$. We conclude that $m\in N_{T^{j_2^l-j_1^l}}(U,V_l)$ for every $1\leq l\leq n$.
\end{proof}

As an application of the above theorems, we exhibit now a non weakly mixing operator $T$ such that $T\oplus T^2\ldots \oplus T^n$ is hypercyclic for every $n$.
\begin{theorem}\label{product hypercyclic pero no weakly mixing}
There exists a non weakly mixing operator $T$ such that $T\oplus T^2\oplus \ldots \oplus T^n$ is hypercyclic for every $n$.
\end{theorem}
 \begin{proof}
 Let $T$ be the operator constructed in Theorem \ref{Sidon non weakly}. Then $T$ is not weakly mixing and $\{T^{j}:j\in J\}$ is disjoint hypercyclic for every finite $J$ such that $J\cup \{0\}$ is Sidon.

Suppose that  $J_n=\{k_1,k_1+1,k_2,k_2+2,\dots,k_n,k_n+n\}$ is a Sidon set. Then Proposition \ref{d-hiper implica sumas hiper} implies that $T\oplus T^2\dots \oplus T^n$ is hypercyclic. Indeed, we may just take $j_2^l=k_l+l$ and $j_1^l=k_l$ for $l=1,\dots, n$.

If we take $k_1=n+1$ and $k_{j+1}=2(k_j+j)+1$ then it is simple to show that $J_n\cup\{0\}$ is Sidon. Indeed, suppose that $0\leq a_1\leq a_2\leq a_3<a_4$ are element in $J_n\cup\{0\}$ such that $a_1+a_4=a_2+a_3$. Notice that, by construction, if $l_1\leq l_2<l_3$ then $k_{l_1}+l_1+k_{l_2}+l_2<k_{l_3}$. This implies that there is $l\leq n$ such that $a_4=k_l+l$ and $a_3=k_l$. Hence, $l=a_2-a_1$. It follows that there must be $l'$ such that $a_2=k_{l'}+l'$ and $a_1=k_{l'}$, because otherwise $a_2-a_1>n+1>l$. Thus, $l=a_2-a_1=l'$, which is a  contradiction because $a_2<a_4.$ 

\end{proof}
\section{Syndetically and frequently transitive operators}\label{syndetically transitive}

In this section we study  syndetically and frequently transitive operators. The motivation comes from the facts that syndetically transitive operators satisfy that $T\oplus T^2\oplus\ldots  \oplus T^n$ is hypercyclic for every $n$ while frequently hypercyclic operators are syndetically transitive. 
In Theorem \ref{caract syndet}
we prove  that a linear operator $T$ is syndetically transitive  if and only if $T\oplus S$ is hypercyclic for every weakly mixing operator $S$. Analougsly, we prove that a linear operator is piecewise syndetically transitive if and only if $T\oplus S$ is hypercyclic for every syndetically transitive operator $S$. In Theorem \ref{frequently transitive non weakly mixing} we show an example of a frequently transitive operator that is not weakly mixing. This answers a question of \cite[Question 5.12]{BesMen19}.

Given a hereditary upward family $F \sub \mathcal P(\zN)$ (also called Furstenberg family) we say that an operator
is $\F$-hypercyclic if there is $x\in X$ for which the sets $N_T (x, U)$ of return times
belong to $\F$ and we say that an operator is $\F$-transitive provided that the sets $N_T (U,V)$ belong to $\F$. 

 We will consider the following families:

\begin{itemize}
    \item $A$ is said to have positive lower density  (or $A\in \underline {\mathcal D}$) if $ \underline{d}(A):\liminf_n \f{\#\{j\leq n:j\in A\}}{n}>0$,
    \item $A$ is said to be thick, provided that $A$ contains arbitrary long intervals,
    \item $A$ is said to be syndetic provided that $A$ has bounded gaps.
    \item $A$ is said to be piecewise syndetic provided that $A$ is the intersection of a thick set with a syndetic set. Equivalently, there is $b$ such that $A$ contains arbitrarily large sets with gaps bounded by $b$ and
    \item $A$ is said to be thickly syndetic provided that for every $k$ there is a syndetic set $S$ such that $S+\{0,\ldots k\}\sub A$.
\end{itemize}

The $\underline {\mathcal D}$-hypercyclic  (transitive) operators are known as frequently hypercyclic (transitive) operators.

Given a family $\F$, the dual family $\F^*$ is defined as 
$$\F^*=\{A\sub \zN: A\cap F\neq \emptyset \text{ for every } F\in\F\}.$$

 The dual of the thick sets (piecewise syndetic sets) are the syndetic sets (thickly syndetic sets) respectively.

It is not difficult to prove that there are no thickly or syndetically hypercyclic operators (See Propositions 2 and 3 of \cite{Bes16} for  proof of these facts). However, operators can be thickly transitive and syndetically transitive. It is well known that a linear operator is thickly transitive if and only if it is weakly mixing \cite[Theorem 4.6]{BayMat09} and that if $A$ has positive lower density, then $A- A$ is syndetic \cite[Proposition 3.19]{Fur81}. This implies that frequently hypercyclic operators are syndetically transitive. On the other hand, the proof of \cite[Theorem 6.31]{BayMat09} shows that syndetically transitive operators are weakly mixing. Thus, frequently hypercyclic operators are both syndetically and thickly transitive. 

For more on $\F$-hypercyclicity see \cite{Bes16,BesMen19,BonGro18,bonilla2020frequently,cardeccia2020arithmetic,cardeccia2021multiple, ErnEssMen21}.

It was proved independently in \cite[Proposition 4]{Moo10} and \cite{GroPer10} that whenever $T$ and $S$ are syndetically and thickly transitive, then $T\oplus S$ is syndetically and thickly transitive. In particular, the finite product of syndetically transitive operators is weakly mixing. 
\begin{theorem}\label{Mooth}
Let $f:X\to X$ and $g:Y\to Y$ be syndetically and thickly transitive continuous mappings. Then $f\oplus g$ is syndetically and thickly transitive.
\end{theorem}

On the other hand, syndetically and transitive operators are thickly syndetically transitive. See \cite[Lemma 2.3]{BesMen19} for a proof of this result. 

\begin{lemma}\label{thickly syndetic}
Let $f:X\to X$ be a syndetically and thickly transitive mapping. Then, $f$ is thickly syndetically transitive.
\end{lemma}

The intersection of a syndetic set and a thick set is always nonempty. Therefore, if $T$ is syndetically transitive and $S$ is weakly mixing then $T\oplus S$ is hypercyclic. This property characterizes the syndetically transitive operators.  

\begin{theorem}\label{caract syndet}
Let $T$ be a linear operator. The following are equivalent:
\begin{enumerate}
    \item $T$ is syndetically transitive.
    \item $T\oplus S$ is weakly mixing for every weakly mixing operator $S$. 
    \item $T\oplus S$ is hypercyclic for every weakly mixing operator $S$.
    \end{enumerate}
\end{theorem}
\begin{proof}

(1)$\Rightarrow$(2). 
Since $T$ is a syndetically transitive linear operator, it is thickly transitive. It follows by the above lemma that $T$ is thickly syndetically transitive. 

Let $U_1,U_2,V_1,V_2$ be nonempty open sets. Then, $N_S(U_2,V_2)$ is thick while $N_T(U_1,V_1)$ is thickly syndetic, i.e. for every $k$, there is a syndetic set $A$ such that $A+\{1,\ldots k\}\sub N_T(U_1,V_1)$. Let $k\in\zN$, then $N_S(U_2,V_2)\cap A+\{1,\ldots k\}$ contains $k$ consecutive integers.
 
(2)$\Rightarrow$ (3) is immediate. 

$(3)\Rightarrow(1)$. We will prove that
if $T$ is not syndetically transitive, then there is a weakly mixing operator $S$ such that $T\oplus S$ is not hypercyclic.

Let $T$ be not syndetically transitive. Hence there are nonempty open sets $U,V$ such that $N_T(U,V)$ is not syndetic. Thus, there is a sequence $n_k$ for which $n_{k+1}>n_k+k$ and 
$$\bigcup_k [n_k,n_k+k]\sub N_T(U,V)^c.$$

It suffices to show the existence of a weakly mixing operator $S$ and a nonempty open set $W$ for which $$N_S(W,W)\sub \bigcup_k [n_k,n_k+k].$$

Let $S$ be the weighted backward shift on $\ell_2$ given by the weights
$$w_n=\begin{cases}
 2 &\text{ if } n\in [n_k,n_k+k] \text{ for some } k;\\
2^{-k} &\text{ if } n=n_k+k+1;\\
1 &\text{ else }.
\end{cases}
$$

The weights $(w_n)_n$ are bounded and $\prod_{j=1}^{n_k+k} w_j=2^k$. These facts imply that $B_w$ is well defined and that $B_w$ is weakly mixing (see \cite[Chapter 4]{GroPer11}). On the other hand we notice that, if $n\notin \bigcup_k [n_k,n_k+k]$, then $\prod_{j=1}^n w_j=1$.

Consider now $W=B(e_1,\varepsilon)$ with $\varepsilon<\f{1}{2}$ and let $n\in N_S(W,W).$ Hence, there is $x$ such that $\|x-e_1\|<\varepsilon$ and $\|S_w^n(x)-e_1\|<\varepsilon.$
Therefore,
$$|x_n|<\varepsilon \ \ \text{  and  }\ \ \left|\prod_{j=1}^n w_j x_n-1\right|=\left|[B_w^n(x)-e_1]_1\right|<\varepsilon.$$ 
It follows that $n\in \bigcup_k [n_k,n_k+k].$ 
\end{proof}

The symmetric problem recovers the piecewise syndetically transitive operators. 
\begin{theorem}\label{caract piece}
Let $T$ be a linear operator. The following are equivalent.
\begin{enumerate}
    \item $T$ is piecewise syndetically transitive.
    \item $T\oplus S$ is hypercyclic for every syndetically transitive operator.
\end{enumerate}
\end{theorem}

\begin{proof}
(1)$\Longrightarrow (2)$.
Suppose that $T$ is piecewise syndetically transitive ans let $S$ be a syndetically transitive operator. Then $S$ is, by Lemma \ref{thickly syndetic}, thickly syndetically transitive. By the duality between the piecewise syndetic and the thickly syndetic sets, we obtain that for every tuple of nonempty open sets $U_1,U_2,V_1,V_2$, $$N_T(U_1,V_1)\cap N_S(U_2,V_2)\neq \emptyset$$
and hence $T\oplus S$ is hypercyclic. 

(2)$\Longrightarrow$ (1).
 Suppose that $T$ is not piecewise syndetically transitive. It suffices to show the existence of a syndetically transitive operator $S$ and nonempty open sets $U,V,W$ such that $N_S(W,W)\sub N_T(U,V)^c$. 

Since $T$ is not piecewise syndetically transitive there are nonempty open sets $U,V$ such that $N_T(U,V)$ is not piecewise syndetic. By duality, there is a thickly syndetic set $A$ such that $N_T(U,V)\cap A=\emptyset$. It follows that $N_T(U,V)^c$ is thickly syndetic. Equivalently, we have that for every $k$, $\{n:[n,n+k]\sub N_T(U,V)^c\}$ is syndetic. 


Let $S$ be the weighted backward shift on $\ell_2$ given by the weights
$$w_n=\begin{cases}
 2 &\text{ if } n\in N_T(U,V)^c;\\
2^{-l} &\text{ if } n\notin N_T(U,V)^c,  n-1\in N_T(U,V)^c \text{ and } l=\max\{j\in\zN
:[n-j,n-1] \sub N_T(U,V)^c\};\\
1 &\text{ else }.
\end{cases}
$$

The weights $w_n$ are bounded and hence $S$ is a well defined backward shift. Also, for every $k$, $\{n:[n,n+k+1]\sub N_T(U,V)^c\}\sub\{n: \prod_{j=1}^n w_j>2^k\}$. This implies that for every $M$, $\{n: \prod_{j=1}^n w_j>M\}$ is syndetic and hence $S$ is sindetically transitive \cite[Corollary 3.4]{BesMen19}. On the other hand, we notice that if $n\notin N(U,V)^c$, then $\prod_{j=1}^n w_j=1$.

Consider now $W=B(e_1,\varepsilon)$ with $\varepsilon<\frac{1}{2}$ and let $n\in N_S(W,W)$. Thus, there is $x$ such that $\|x-e_1\|<\varepsilon$ and $\|S^n(x)-e_1\|<\varepsilon$. 
Therefore,
$$|x_n|<\varepsilon \ \ \text{  and  }\ \ \left|\prod_{j=1}^n w_j x_n-1\right|=\left|[B_w^n(x)-e_1]_1\right|<\varepsilon.$$ 
It follows that $n\in N_T(U,V)^c.$
\end{proof}

 Every syndetically transitive linear operator is weakly mixing. It is natural to ask the same for the family of sets having positive lower density. The following Theorem answers a question posed in \cite[Question 5.12]{BesMen19}.
\begin{theorem}\label{frequently transitive non weakly mixing}
There exist a non weakly mixing operator such that $T$ is frequently transitive. 
\end{theorem}
The proof relies again on the construction of Bayart and Matheron of a non weakly mixing but hypercyclic operator. The following Theorem was proved in \cite{BayMat09Non}.
\begin{theorem}[Bayart and Matheron]\label{Bayart Matheron}
Let $(m_k)$ be a sequence of natural numbers such that $\lim_k \f{m_k}{k}=\infty.$ Then there exist a non weakly mixing operator $T$ and a vector $x$ such that for every nonempty open set $U$, $N_T(x,U)\in O(m_k).$
\end{theorem}
We will need the following lemma, which is number-theoretic. 
\begin{lemma}
Let $(n_k)$ be an increasing sequence of natural numbers in $O(n^2)$. Then the set $\{n_k-n_j:k\geq j\}$ has positive lower density.
\end{lemma}
\begin{proof}
Let $C>0$ such that $n_k\leq Ck^2.$ Notice that $\#\{k\leq K:n_{k+1}-n_{k}\geq 8Ck\}\leq \f{K}{2}$. Indeed, if we suppose otherwise, then we get that
\begin{align*}
    n_K-n_1=\sum_{l=1}^K n_{l+1}-n_l\geq \sum_{l=1}^{\f{K}{2}}8Ck>CK^2
\end{align*}
which is a contradiction.
Hence, $\#\{k\leq K:n_{k+1}-n_{k}\leq 8Ck\}\geq \f{K}{2}$. 
This implies that 
$$\underline d \{n_k-n_j:k\geq j\}\geq \f{1}{16C}.$$
\end{proof}

\begin{proof}[Proof of Theorem \ref{frequently transitive non weakly mixing}]
Let $T$ be a non weakly mixing operator and $x$ such that for every nonempty open set $U$, $N_T(x,U)\in O(n^2)$. By the above lemma we get that for every nonempty open set $U$, $N_T(U,U)=N_T(x,U)- N_T(x,U)$ has positive lower density.

Consider now a pair of nonempty open sets $U,V$. Since $T$ is hypercyclic, there are $U'\sub U$ and $n$ such that $T^n(U')\sub V$. Hence $N_T(U',U')+n\sub N_T(U,V)$ and therefore $N_T(U,V)$ has positive lower density.
\end{proof}

We would like to end the section with an open question related to Theorems \ref{caract syndet} and \ref{caract piece}.

\textbf {Question.} Let $T$ be a piecewise syndetically linear operator. Is $T$ weakly mixing?

\begin{remark}
 The techniques used in Theorems \ref{Sidon non weakly} and \ref{frequently transitive non weakly mixing} do not provide non weakly mixing but piecewise syndetically transitive operators. Indeed, Bayart and Matheron construction relies on the existence of a $\Delta_l$-Sidon sequence $b_n$ such that for every nonempty open set $U$, there is $(b_{n_k})_k$ such that $b_{n_k}\in N(e_1,U)$ for every $k$. However, the $\Delta_l$-Sidon structure of $b_n$ implies that the set of differences $\{b_{n}-b_{n'}:n>n'\}$ is not piecewise syndetic. In fact, from this observation and the proof of Theorem \ref{Bayart Matheron} it follows that given a sequence of natural numbers $(m_k)_k$ with $\f{m_k}{k}\to \infty$, there is $n_k\in O(m_k)$ such that the set of differences $\{n_k-n_{k'}:k>k'\}$ is not piecewise syndetic.
\end{remark}

\end{document}